\documentclass[english,11pt]{article}
\usepackage{stmaryrd}
\usepackage{tipa}
\usepackage{tikz}
\usepackage{amsmath}
\usepackage{amssymb}
\usepackage{pst-node}
\usepackage{color}
\usepackage{extarrows}
\usepackage{graphicx,xcolor,float}
\usepackage{epstopdf}
\usetikzlibrary{arrows,shapes}
\usepackage{verbatim}

\newtheorem{theorem}{Theorem}[section]
\newtheorem{corollary}[theorem]{Corollary}

\newtheorem{problem}[theorem]{Problem}
\newtheorem{lemma}[theorem]{Lemma}

\newenvironment{proof}{\noindent {\bf Proof.}}{\rule{3mm}{3mm}\par\medskip}

\newcommand{\ZA}{\mbox{$\mathcal Z$}}
\newcommand{\Z}{\mbox{$\mathbb Z$}}

\begin{document}

\title{Group Connectivity under $3$-Edge-Connectivity}

 \author{
  Miaomiao Han$^1$, Jiaao Li$^2$, Xueliang Li$^3$, Meiling Wang$^3$\\
  {\small $^1$College of Mathematical Science, Tianjin Normal University, Tianjin 300387, China}\\
 {\small $^2$School of Mathematical Sciences and LPMC, Nankai University, Tianjin 300071, China}\\
 {\small$^3$Center for Combinatorics and LPMC, Nankai University, Tianjin 300071, China}\\
 {\footnotesize Emails: mmhan2018@hotmail.com; lijiaao@nankai.edu.cn; lxl@nankai.edu.cn; Estellewml@gmail.com}}

 \date{}
\maketitle

\begin{abstract}
Let $S,T$ be two distinct finite Abelian groups with $|S|=|T|$.
A fundamental theorem of Tutte shows that a graph admits a nowhere-zero $S$-flow
if and only if it admits a nowhere-zero $T$-flow. Jaeger, Linial, Payan and Tarsi in 1992
introduced group connectivity as an extension of flow theory,
and they asked whether such a relation holds for group connectivity analogy.
It was negatively answered by Hu\v{s}ek,  Moheln\'{i}kov\'{a} and \v{S}\'{a}mal in 2017 for graphs with edge-connectivity 2
for the groups $S=\Z_4$ and $T=\Z_2^2$. In this paper, we extend their results
to $3$-edge-connected graphs (including both cubic and general graphs),
which answers open problems proposed by Hu\v{s}ek, Moheln\'{i}kov\'{a} and \v{S}\'{a}mal(2017) and Lai, Li, Shao and Zhan(2011).
Combining some previous results, this characterizes all the equivalence
of group connectivity under $3$-edge-connectivity, showing that every
$3$-edge-connected $S$-connected graph is $T$-connected if and only
if $\{S,T\}\neq \{\Z_4,\Z_2^2\}$.
\\[2mm]
\\[2mm] \textbf{AMS Subject Classification (2010):} 05C21, 05C40, 05C15
\end{abstract}

\section{Introduction}

Graphs considered in this paper are finite and loopless, with possible parallel edges.
Throughout this paper, let $S, T$ be (additive) Abelian groups, and $\Z_k$ the cyclic group of order $k$.
We follow \cite{Bondy} for undefined notation and terminology.
Fix an orientation $D$ of a graph $G$. For any
$x\in V(G)$, let $E^+_D(x)$ ($E^-_D(x)$, resp.) denote the set of all edges
directed away from (into, resp.) $x$. Given a mapping $\varphi: E(G)\mapsto S$, define, for every vertex
$u\in V(G)$,
 \[
\partial\varphi(u)=\sum\limits_{e\in E^+_D(u)}\varphi(e) ~-\sum\limits_{e\in E^-_D(u)}\varphi(e).
\]
Evidently, we have $\sum_{u\in V(G)}\partial\varphi(u)= 0$ since each directed edge is counted exactly
once in both its head and tail. A {\bf zero-sum boundary function} is a mapping $\gamma: V(G) \mapsto S$
satisfying $\sum_{u\in V(G)}\gamma(u)= 0$, which is necessary for the existence of such mapping $\varphi$
with $\partial\varphi=\gamma$. Let {\bf $\ZA(G, S)$} denote the collection of all zero-sum boundary functions
of $G$. A group flow, {\bf $S$-flow}, of $G$ is a mapping $\varphi: E(G)\mapsto S$ with $\partial\varphi=\bf{0}$,
where ${\bf 0}\in \ZA(G, S)$ denotes the constant zero mapping. If $\varphi(e)\neq  0$ for each edge $e\in E(G)$,
then $\varphi$ is called a {\bf nowhere-zero $S$-flow}, abbreviated as $S$-NZF.
When $S=\Z$ and $0<|\varphi(e)|<k$ for any $e\in E(G)$, it is known as a {\bf nowhere-zero $k$-flow}, abbreviated as $k$-NZF.

The flow theory was initiated by Tutte \cite{Tutt54} in studying face coloring problems of graphs
on the plane and other surfaces. Tutte \cite{Tutt54} proposed some flow conjectures, which are considered as core problems in graph theory.
Tutte's $3$-flow and $5$-flow conjectures predict the existence of flow for given edge-connectivity $4$ and $2$, respectively,
regardless the topological embedding structures of graphs. The $4$-flow conjecture \cite{Tutt66}, generalizing the celebrated Four Coloring Theorem,
asserts every Petersen-minor-free graph admits a $4$-NZF.
Those problems are widely studied and remain open, while significant progress have been made by Jaeger \cite{Jaeger1979},
Seymour \cite{Seym1981}, Thomassen \cite{Thomassen2012}, and Lov\'asz et al. \cite{LTWZ2013}. We refer to \cite{LLZ15} for a recent survey on those topics.
One of the critical tools in studying nowhere-zero flows is the following fundamental theorem of Tutte \cite{Tutt66}, converting group flows into integer flows.
\begin{theorem}\cite{Tutt66}
A graph admits a $k$-NZF if and only if it admits an $S$-NZF for some Abelian group $S$ with $|S|=k$.
\label{Thm:Tutte66}
\end{theorem}

The advantage of group flows is to provide much more flexibility in proving related integer flow theorems,
which allows to use certain contraction operations and local adjustments on graphs.
To facilitate this approach, Jaeger et al. \cite{JLPT92} introduced group connectivity concept as a generalization of $S$-flow.
If for every $\gamma\in {\ZA}(G, S)$, there is a mapping $\varphi: E(G)\mapsto S\setminus\{0\}$ such that $\partial\varphi=\gamma$,
then $G$ is called {\bf $S$-connected}. Due to certain stronger conditions in group connectivity, some nice properties of flows
can not be easily extended to group connectivity. For example, the monotonicity fails for group connectivity.
It follows from the definition that every $k$-NZF admissible graph has a $(k+1)$-NZF,
and so by Theorem \ref{Thm:Tutte66} every $T$-NZF admissible graph has an $S$-NZF for any finite Abelian groups $S,T$ with $|S|\ge |T|$.
However, Jaeger et al. \cite{JLPT92} showed that there exist $\Z_5$-connected graphs which are not $\Z_6$-connected,
and similar examples were exhibited for some other large groups of prime order.
On the positive side, an unusual monotonicity of group connectivity was proved in \cite{LLL2017}
that every $\Z_3$-connected graph is $S$-connected for $|S|\ge 4$.

For two distinct finite Abelian groups $S, T$ with the same order, Jaeger et al. \cite{JLPT92} asked
whether $S$-connectivity and $T$-connectivity are equivalent, similar as Theorem \ref{Thm:Tutte66},
and they remarked that it is even unknown for the first case concerning $\Z_4$ and $\Z_2^2$.
Lai et al. \cite{LLSZ11} further proposed the problem below for $3$-edge-connected graphs.

\begin{problem}\label{lai}{\em (Problem 1.8 in Lai et al. \cite{LLSZ11})}
Let ${\mathcal F}(S)$ be the family of all $3$-edge-connected $S$-connected graphs.
Is it true that for two Abelian groups $S_1$ and $S_2$, if $|S_1|=|S_2|$, then
$${\mathcal F}(S_1)={\mathcal F}(S_2)?$$
\end{problem}

With a computer-aided approach,  Hu\v{s}ek, Moheln\'{i}kov\'{a} and \v{S}\'{a}mal \cite{Husek}
constructed $2$-edge-connected graphs to show that $\Z_4$-connectivity and $\Z^2_2$-connectivity
are not equivalent and obtained the following theorem, which provides a negative
answer to the question of Jaeger et al. \cite{JLPT92}.

\begin{theorem}\cite{Husek}\label{2edgecon} Denote by $H_1, H_2$ as the graphs depicted in Figure \ref{fig2connected}.

(1) \ The graph $H_1$ is $\Z_2^2$-connected but not $\Z_4$-connected.

(2) \ The graph $H_2$ is $\Z_4$-connected but not $\Z_2^2$-connected.

\noindent Furthermore, infinitely many such examples can be constructed by replacing some vertices with triangles repeatedly.
\end{theorem}

\begin{figure}[ht]
\setlength{\unitlength}{0.10cm}

\begin{center}

\begin{picture}(60,20)

\put(0,0){\circle*{1.2}}\put(-10,0){\circle*{1.2}}\put(-10,10){\circle*{1.2}}\put(10,10){\circle*{1.2}}\put(10,-10){\circle*{1.2}}\put(0,-15){\circle*{1.2}}\put(-10,-10){\circle*{1.2}}
\put(0,15){\circle*{1.2}}\put(15,10){\circle*{1.2}}\put(15,-10){\circle*{1.2}}\put(0,-20){\circle*{1.2}}\put(-15,-10){\circle*{1.2}}\put(-15,10){\circle*{1.2}}\put(-15,0){\circle*{1.2}}
\put(5,-12.5){\circle*{1.2}}

\qbezier(-10,10)(0,10)(10,10)\qbezier(10,-10)(10,0)(10,10)\qbezier(10,-10)(5,-12.5)(0,-15)\qbezier(-10,-10)(-5,-12.5)(0,-15)\qbezier(-10,-10)(-10,-5)(-10,0)\qbezier(-10,0)(-10,5)(-10,10)
\qbezier(0,15)(0,7.5)(0,0)\qbezier(0,0)(2.5,-6.25)(5,-12.5)\qbezier(10,10)(12.5,10)(15,10)\qbezier(10,-10)(12.5,-10)(15,-10)\qbezier(0,-15)(0,-17.5)(0,-20)\qbezier(-10,-10)(-12.5,-10)(-15,-10)
\qbezier(-15,10)(-12.5,10)(-10,10)\qbezier(0,15)(7.5,12.5)(15,10)\qbezier(15,10)(15,0)(15,-10)\qbezier(15,-10)(7.5,-15)(0,-20)\qbezier(0,-20)(-7.5,-15)(-15,-10)\qbezier(-15,-10)(-15,0)(-15,10)
\qbezier(-15,10)(-7.5,12.5)(0,15)

\put(50,0){\circle*{1.2}}\put(40,0){\circle*{1.2}}\put(40,10){\circle*{1.2}}\put(60,10){\circle*{1.2}}\put(60,-10){\circle*{1.2}}\put(50,-15){\circle*{1.2}}\put(40,-10){\circle*{1.2}}
\put(50,15){\circle*{1.2}}\put(65,10){\circle*{1.2}}\put(615,-10){\circle*{1.2}}\put(50,-20){\circle*{1.2}}\put(35,-10){\circle*{1.2}}\put(35,10){\circle*{1.2}}\put(42.5,-15){\circle*{1.2}}
\put(55,-12.5){\circle*{1.2}}
\qbezier(40,10)(50,10)(60,10)\qbezier(60,-10)(60,0)(60,10)\qbezier(60,-10)(55,-12.5)(50,-15)\qbezier(40,-10)(45,-12.5)(50,-15)\qbezier(40,-10)(40,-5)(40,0)\qbezier(40,0)(40,5)(40,10)
\qbezier(50,15)(50,7.5)(50,0)\qbezier(50,0)(52.5,-6.25)(55,-12.5)\qbezier(60,10)(62.5,10)(65,10)\qbezier(60,-10)(62.5,-10)(65,-10)\qbezier(50,-15)(50,-17.5)(50,-20)\qbezier(40,-10)(37.5,-10)(35,-10)
\qbezier(35,10)(37.5,10)(40,10)\qbezier(50,15)(57.5,12.5)(65,10)\qbezier(65,10)(65,0)(65,-10)\qbezier(65,-10)(57.5,-15)(50,-20)\qbezier(50,-20)(42.5,-15)(35,-10)\qbezier(35,-10)(35,0)(35,10)
\qbezier(35,10)(42.5,12.5)(50,15)

\put(-3,-27){$H_1$}\put(47,-27){$H_2$}

\end{picture}
\vspace{2cm}
\end{center}
\caption{\small\it The graphs for Theorem \ref{2edgecon}.}
\label{fig2connected}
\end{figure}
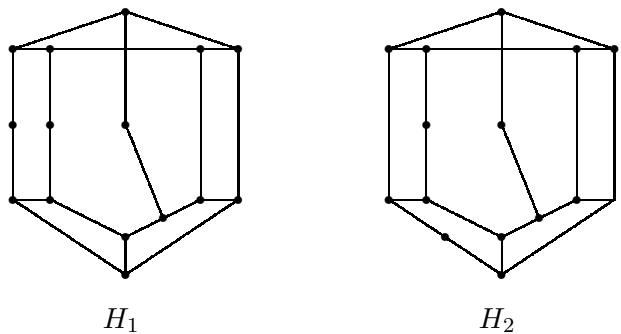

By developing a $2$-sum operation for group connectivity (as defined below), we extend Theorem \ref{2edgecon} to $3$-edge-connected graphs.

\begin{theorem}\label{mainth}
~
\begin{itemize}
\item[(1)] There exists a $3$-edge-connected graph which is $\Z_4$-connected but not $\Z_{2}^2$-connected.
\item[(2)] There exists a $3$-edge-connected graph which is $\Z_{2}^2$-connected but not $\Z_4$-connected.
\end{itemize}
Furthermore, infinitely many such graphs can be generated by a number of $2$-sum operations.
\end{theorem}

It is worth noting that our proof of Theorem \ref{mainth} is theoretical, although it assumes the truth of Theorem \ref{2edgecon} (whose proof is computer-aided).

Extending Jaeger's $4$-flow theorem and Seymour's $6$-flow theorem, Jaeger et al. \cite{JLPT92} obtained the following group connectivity analogy.
\begin{theorem}\label{thmJLPT}\cite{JLPT92}
(i) \ Every $4$-edge-connected graph is $S$-connected for $|S|\ge 4$.\\
(ii) \ Every $3$-edge-connected graph is $S$-connected for $|S|\ge 6$.
\end{theorem}

Combining Theorems \ref{mainth} and \ref{thmJLPT}, we immediately have the following corollary,
characterizing the equivalence of group connectivity for all $3$-edge-connected graphs completely.
This answers Problem \ref{lai}.

\begin{corollary}\label{Cor:3connn}
Let $S,T$ be two distinct Abelian groups with $|S|=|T|$. Then every $3$-edge-connected $S$-connected graph
is $T$-connected if and only if $\{S,T\}\neq \{\Z_4,\Z_2^2\}$.
\end{corollary}

In \cite{Husek}, Hu\v{s}ek et al. also asked whether such $3$-edge-connected cubic graphs exist.
In fact, Theorem \ref{mainth} was obtained in early 2018, and the second author communicated with
Robert \v{S}\'{a}mal in SIAM Conference on Discrete Mathematics, Denver, June 2018.
The existence of such $3$-edge-connected cubic graphs was still open for a while, see Section 5
in Hu\v{s}ek et al. \cite{Husek}. Now we are able to solve it by a new construction method.

\begin{theorem}\label{mainthmcubic}
~
\begin{itemize}
\item[(1)] There exists a $3$-edge-connected cubic graph which is $\Z_4$-connected but not $\Z_{2}^2$-connected.
\item[(2)] There exists a $3$-edge-connected cubic graph which is $\Z_{2}^2$-connected but not $\Z_4$-connected.
 \end{itemize}
Moreover, infinitely many such graphs can be constructed by substituting some vertices with triangles repeatedly.
\end{theorem}

The paper is organized as follows. In Section \ref{sec:2sum}  we first develop a $2$-sum operation for group connectivity
and use it to prove Theorem \ref{mainth}. Then in Section \ref{sec:cubic} we apply a new method to construct such
cubic graphs through flow properties of two special graphs.
In Section \ref{sec:collapsible}, we end this paper with a few concluding remarks.

\section{Constructions via $2$-sum operations}
\label{sec:2sum}

For $1\le i\le 2$, let $\Gamma_i$ be a graph with two distinct vertices $u_i,v_i\in V(\Gamma_i)$.
If $u_1v_1\in E(\Gamma_1)$, then we define $\Gamma=\Gamma_1(u_1v_1)\oplus \Gamma_2(u_2,v_2)$,
called the {\bf $2$-sum} of $\Gamma_1$ and $\Gamma_2$, as the graph obtained from $\Gamma_1$ and $\Gamma_2$
by removing the edge $u_1v_1$ in $\Gamma_1$, and then identifying $u_1$ and $u_2$ as a new vertex
$u$, and identifying $v_1$ and $v_2$ as a new vertex $v$ (see Figure \ref{FIG:2sum}).

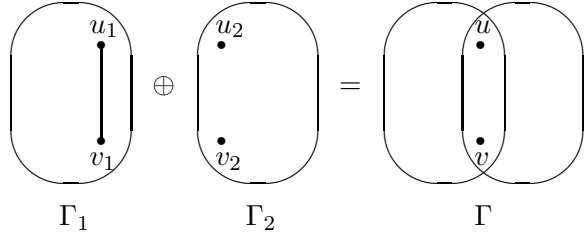
\begin{figure}[ht]
\setlength{\unitlength}{0.08cm}

\begin{center}

\begin{picture}(125,40)

\put(20, 20){\oval(20, 30)}

\put(51, 20){\oval(20, 30)}

\put(18,-2){$\Gamma_1$}\put(49,-2){$\Gamma_2$}

\put(25,28){\circle*{1.2}}
\put (23, 30){$u_1$}

\put(25,12){\circle*{1.2}}
\put(23,8){$v_1$}

\qbezier(25,12)(25,20)(25,28)

\put(45,28){\circle*{1.2}}
\put (44, 30){$u_2$}

\put(45,12){\circle*{1.2}}
\put(44,8){$v_2$}
\put(33.5,20){$\oplus$}

\put(82, 20){\oval(20, 30)}

\put(95, 20){\oval(20, 30)}

\put(88,28){\circle*{1.2}}
\put (87, 30){$u$}

\put(88,12){\circle*{1.2}}
\put(87,8){$v$}
\put (87, -2){$\Gamma$}

\put(64.5,20){$=$}
\end{picture}
\end{center}
\caption{\small\it The $2$-sum $\Gamma=\Gamma_1(u_1v_1)\oplus \Gamma_2(u_2,v_2)$.}
\label{FIG:2sum}
\end{figure}

This $2$-sum operation can be viewed as a dual operation of Haj\'{o}s join on graph coloring.
It was first developed by Kochol \cite{Kochol2001} in studying $3$-flow problem,
and later generalized to $\Z_3$-connectivity in \cite{HanLL18}.
Here we extend this $2$-sum property to group connectivity of arbitrary finite Abelian groups.
\begin{lemma}\label{2sum}
Let $S$ be a finite Abelian group with $|S|\ge 3$. If neither $\Gamma_1$ nor $\Gamma_2$ is $S$-connected,
then $\Gamma=\Gamma_1\oplus \Gamma_2$ is not $S$-connected.
\end{lemma}
\begin{proof}
Let $u, v \in V(\Gamma)$ and $u_i, v_i \in V(\Gamma_i)$ where $i=1,2$ as defined above.
That is, $\Gamma=\Gamma_1(u_1v_1)\oplus \Gamma_2(u_2,v_2)$. Since $\Gamma_i$ is not $S$-connected
for each $i\in \{1,2\}$, there exists a $\beta_i\in \ZA(\Gamma_i,S)$ such that for any orientation
of $\Gamma_i$ and any mapping $\varphi_i: E(\Gamma_i)\mapsto S\setminus\{0\}$,
we have $\partial\varphi_i\neq\beta_i$.

For each $z\in V(\Gamma)$, define
$$\varepsilon(z)=
\left\{\begin{array}{ll}
\beta_1(u_1)+\beta_2(u_2) & \mbox{if~} z=u; \\
\beta_1(v_1)+\beta_2(v_2) & \mbox{if~} z=v;\\
\beta_1(z) & \mbox{if~} z\in V(\Gamma_1)\setminus\{u_1,v_1\};\\
\beta_2(z) & \mbox{otherwise}.\\
\end{array}
\right.
$$
It is routine to check that $\sum_{z\in V(\Gamma)}\varepsilon(z)=\sum_{x\in V(\Gamma_1)}\beta_1(x)+\sum_{y\in V(\Gamma_2)}\beta_2(y)=0$, and so $\varepsilon\in\ZA(\Gamma,S)$.

Suppose, on the contrary, that $\Gamma$ is $S$-connected. Fix an orientation $D$ of $\Gamma$.
Then there exists a mapping $\eta: E(\Gamma)\mapsto S\setminus\{0\}$ such that $\partial \eta=\varepsilon$. In particular, we have
\[
\sum\limits_{e\in E^+_D(u)}\eta(e)-\sum\limits_{e\in E^-_D(u)}\eta(e)=\partial \eta(u)=\varepsilon(u)
 \]
 and
 \[
 \sum\limits_{e\in E^+_D(v)}\eta(e)-\sum\limits_{e\in E^-_D(v)}\eta(e)=\partial\eta(v)=\varepsilon(v).
\]

Let $D_2$ be the restriction of $D$ in $\Gamma_2$. Consider $D_2$ and $\eta$ on $\Gamma_2$.
As $\partial\eta(z)=\beta_2(z), \forall z\in V(\Gamma_2)\setminus\{u_2,v_2\}$, we have
\begin{eqnarray*}
  \partial\eta(u_2)+\partial\eta(v_2) &=& 0-\sum\limits_{z\in V(\Gamma_2)\setminus\{u_2,v_2\}}\partial\eta(z)\\
  &=& 0-\sum\limits_{z\in V(\Gamma_2)\setminus\{u_2,v_2\}}\beta_2(z)\\
  &=& \beta_2(u_2)+\beta_2(v_2).
\end{eqnarray*}
Since $\partial\varphi\neq\beta_2$ for any mapping $\varphi: E(\Gamma_2)\mapsto S\setminus\{0\}$,
it follows that $\partial\eta\neq\beta_2$, and so
$\partial\eta(u_2)\neq\beta_2(u_2)$ from the above equation. Thus there exists a nonzero
element $b\in S$ such that $\partial\eta(u_2)= \beta_2(u_2)+b$ and $\partial\eta(v_2)= \beta_2(v_2)-b$ in $\Gamma_2$.

Now consider $\eta$ and $D_1$, the restriction of $D$ on $\Gamma_1-u_1v_1$.
We have
$$\partial\eta(u_1)=\varepsilon(u)-[\beta_2(u_2)+b]=[\beta_1(u_1)+\beta_2(u_2)]-[\beta_2(u_2)+b]=\beta_1(u_1)-b$$
and
$$\partial\eta(v_1)=\varepsilon(v)-[\beta_2(v_2)-b]=\beta_1(v_1)+b.$$

We orient the edge $u_1v_1$ from $u_1$ to $v_1$ in $\Gamma_1$. Together with $D_1$, this gives an orientation $D_1'$ of $\Gamma_1$.
Define a mapping $\omega: E(\Gamma_1)\mapsto S\setminus\{0\}$ such that, for every  $e\in E(\Gamma_1)$,
$$\omega(e)=
\left\{\begin{array}{ll}
b & \mbox{if~} e=u_1v_1; \\
\eta(e) & \mbox{otherwise.}
\end{array}
\right.
$$
Then $\partial \omega(z)=\partial\eta(z)=\beta_1(z),\ \forall z\in V(\Gamma_1)\setminus\{u_1,v_1\}$.
Moreover, $\partial \omega(u_1)=\partial\eta(u_1)+\omega(u_1v_1)=\beta_1(u_1)$ and
$\partial \omega(v_1)=\partial\eta(v_1)-\omega(u_1v_1)=\beta_1(v_1)$.
Conclude that $\partial \omega=\beta_1$, which is a contradiction.
\end{proof}

For $X\subseteq E(G)$, the {\bf contraction} $G/X$ is the graph
obtained by identifying the two ends of each edge in $X$
and then deleting the resulting loops from $G$. If $H$ is a subgraph of $G$,
$G/H$ is used to represent $G/E(H)$ for short.
For proving $S$-connectivity, the following lemma would be helpful.

\begin{lemma}\cite{LaiH00}\label{con}
(1) \ A cycle $C_n$ of length $n$ is $S$-connected if and only if
$|S| \ge n+1$.\\
(2) \ If $H$ is an $S$-connected subgraph of a graph $G$, then $G$ is $S$-connected if and only if $G/H$ is $S$-connected.
\end{lemma}

A vertex of degree $k$ is called a $k$-vertex.
Let $C_4$ be a $4$-cycle with $V(C_4)=\{v_4,v_3,v_2,v_1\}$.
Fix $i\in\{1,2\}$. In Figure \ref{fig2connected}, observe that there are exactly three $2$-vertices,
denoted by $x_i,y_i,z_i$ in $H_i$. Attach two copies of $H_i$, namely $H_i$ and $H_i'$
(whose corresponding $2$-vertices are $x_i',y_i',z_i'$). Let $H_{i}^1$ be the graph obtained from
$C_4$ and $H_i$  by the $2$-sum operation on $v_1v_2$ and $x_i, y_i$, namely $H_{i}^1=C_4(v_1v_2)\oplus H_i(x_i,y_i)$.
Construct a graph $H_{i}^2$ from $H_{i}^1$ and $H_i'$ by the $2$-sum operation on $v_4v_3$ and $x_i', y_i'$,
that is, $H_{i}^2=H_i^1(v_4v_3)\oplus H_i'(x_i',y_i')$. See Figure \ref{FIG: H1_2} for the construction of $H_1^2$.

\begin{figure}[ht]
\setlength{\unitlength}{0.08cm}

\begin{center}

\begin{picture}(120,35)

\put(20, 20){\oval(10, 25)}

\put(110, 20){\oval(10, 25)}

\put(16,18){$H_1$}\put(108,18){$H_1'$}

\qbezier(23,22)(26.5,21)(30,20)
\qbezier(23,18)(26.5,19)(30,20)
\qbezier(23,28)(26.5,27)(30,26)
\qbezier(23,24)(26.5,25)(30,26)
\qbezier(23,16)(26.5,15)(30,14)
\qbezier(23,12)(26.5,13)(30,14)

\put(30,20){\circle*{1.0}}
\put(30,26){\circle*{1.0}}
\put(30,14){\circle*{1.0}}

\put(100,20){\circle*{1.0}}
\put(100,26){\circle*{1.0}}
\put(100,14){\circle*{1.0}}

\put(60,25){\circle*{1.0}}
\put(60,15){\circle*{1.0}}
\put(70,15){\circle*{1.0}}
\put(70,25){\circle*{1.0}}

\qbezier(60,25)(65,25)(70,25)
\qbezier(60,25)(60,20)(60,15)
\qbezier(60,15)(65,15)(70,15)
\qbezier(70,15)(70,20)(70,25)

\put (43, 20){$\oplus$}
\put (85, 20){$\oplus$}

\qbezier(107,22)(103.5,21)(100,20)
\qbezier(107,18)(103.5,19)(100,20)
\qbezier(107,28)(103.5,27)(100,26)
\qbezier(107,24)(103.5,25)(100,26)
\qbezier(107,16)(103.5,15)(100,14)
\qbezier(107,12)(103.5,13)(100,14)

\put (95, 25){$x_1'$}
\put (95, 19){$y_1'$}
\put (95, 13){$z_1'$}
\put (31, 25){$x_1$}
\put (31, 19){$y_1$}
\put (31, 13){$z_1$}
\put (55, 24){$v_1$}
\put (55, 15){$v_2$}
\put (71, 24){$v_4$}
\put (71, 15){$v_3$}

\end{picture}
\end{center}
\vspace{-0.6cm}
\caption{\small\it The graph $H_1^2$ in Lemma \ref{h12}.}
\label{FIG: H1_2}
\end{figure}
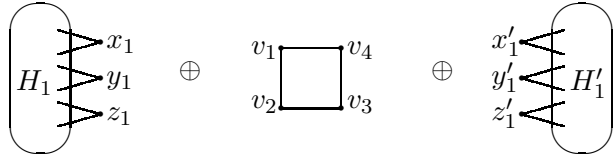

\begin{lemma}\label{h12}
(1) \ The graph $H_{1}^2$ is $\Z_2^2$-connected, but not $\Z_4$-connected.\\
(2) \ The graph $H_{2}^2$ is $\Z_4$-connected, but not $\Z_2^2$-connected.
\end{lemma}
\begin{proof}
(1) \ By Theorem \ref{2edgecon}, $H_1$ and $H_1'$ are $\Z_2^2$-connected. Notice that
$$(H_{1}^2/H_1)/H_1'=(C_4/v_1v_2)/v_3v_4=C_2,$$
which is $\Z_2^2$-connected. By Lemma \ref{con} we see that $H_{1}^2/H_1$ is $\Z_2^2$-connected.
As $H_1$ is $\Z_2^2$-connected and by Lemma \ref{con} again, $H_1^2$ is $\Z_2^2$-connected as desired.
Since $H_1^1=C_4(v_1v_2)\oplus H_1(x_1,y_1)$ is obtained from the $2$-sum of two non-$\Z_2^2$-connected graphs $C_4$ and $H_1$,
we know that $H_1^1$ is not $\Z_2^2$-connected by Lemma \ref{2sum}.
Similarly, as $H_{1}^2=H_1^1(v_4v_3)\oplus H_1'(x_1',y_1')$, where neither $H_1^1$ nor $H_1'$ is $\Z_2^2$-connected,
it follows from Lemma \ref{2sum} that $H_1^2$ is not $\Z_2^2$-connected either.

(2) \ The proof is very similar to (1). Since $H_{2}$ is $\Z_4$-connected, but not $\Z_2^2$-connected,
after applying the $2$-sum operation twice, the resulting graph $H_2^2$ is $\Z_4$-connected by Lemma \ref{con},
but not $\Z_2^2$-connected by Lemma \ref{2sum}.
\end{proof}

\begin{figure}[ht]\label{FIG: 2sumH13}
\setlength{\unitlength}{0.1cm}

\begin{center}

\begin{picture}(140,30)

\put(35, 20){\oval(10, 15)}

\put(95, 20){\oval(10, 15)}

\put(32, 20){$H_{1}^2$}

\put(93, 20){$H_{1}^2$}
\put(63, -4){$H_{1}^2$}

\put(95, 20){\oval(10, 15)}
\put(65, -3){\oval(15, 10)}

\qbezier(60,25)(65,25)(70,25)
\qbezier(60,25)(60,20)(60,15)
\qbezier(60,15)(65,15)(70,15)
\qbezier(70,15)(70,20)(70,25)

\put(70,25){\circle*{1.0}}\put(71,24){$v_4$}
\put(60,25){\circle*{1.0}}\put(55,24){$v_1$}
\put(70,15){\circle*{1.0}}\put(71,15){$v_3$}
\put(60,15){\circle*{1.0}}\put(55,15){$v_2$}

\put(45,24){\circle*{1.0}}\put(44,26){$z_1$}
\put(45,16){\circle*{1.0}}\put(44,13){$z_1'$}

\put (50, 19){$\oplus$}
\put (75, 19){$\oplus$}
\put (63.5,10){$\oplus$}

\put(61,7){\circle*{1.0}}\put(57,7){$z_2$}
\put(69,7){\circle*{1.0}}\put(70,7){$z_2'$}

\qbezier(45,24)(41.5,23)(38,22)
\qbezier(45,24)(41.5,25)(38,26)
\qbezier(45,16)(41.5,17)(38,18)
\qbezier(45,16)(41.5,15)(38,14)

\qbezier(85,24)(88.5,23)(92,22)
\qbezier(85,24)(88.5,25)(92,26)
\qbezier(85,16)(88.5,17)(92,18)
\qbezier(85,16)(88.5,15)(92,14)

\put(85,24){\circle*{1.0}}\put(84,26){$z_3$}
\put(85,16){\circle*{1.0}}\put(84,13){$z_3'$}

\qbezier(63,0)(62,3.5)(61,7)
\qbezier(59,0)(60,3.5)(61,7)
\qbezier(67,0)(68,3.5)(69,7)
\qbezier(71,0)(70,3.5)(69,7)

\end{picture}
\end{center}
\vspace{0.5cm}
\caption{\small\it $H_1^3:$ Graph of Theorem \ref{THM: H13H23} (1).}
\label{FIG: 2sumH13}
\end{figure}
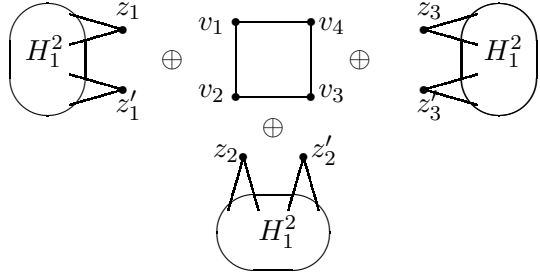

Note that, by the construction above, the graph $H_{i}^2$, for each $i\in\{1,2\}$, has precisely two vertices $z_i$ and $z_i'$ of degree two.
Now we would construct $H_i^3=C_4\oplus H_i^2\oplus H_i^2\oplus H_i^2$, $i\in\{1,2\}$, that would be used in the following theorem.
The way to construct $H_2^3$ from $H_2^2$ is the same as constructing $H_1^3$ from  $H_1^2$. So we take $H_1^3$ as an example.
Attach three copies of $H_1^2$, whose $2$-vertices are denoted by $z_1,z_1'$, $z_2,z_2'$ and $z_3,z_3'$, respectively.
Apply the $2$-sum operation three times on $C_4$ and the copies of $H_1^2$. Specifically, we first apply $2$-sum on the edge $v_1v_2$
with $z_1,z_1'$ in the first copy of $H_1^2$, then apply $2$-sum on the edge $v_2v_3$ with $z_2,z_2'$ in the second copy,
and apply the last $2$-sum on the edge $v_3v_4$ with $z_3,z_3'$ in the third copy, as demonstrated in Figure \ref{FIG: 2sumH13}.
This gives  the resulting graph $H_1^3$.

\begin{theorem}\label{THM: H13H23}
(1) \ The graph $H_{1}^3$ is $3$-edge-connected, $\Z_2^2$-connected, but not $\Z_4$-connected.\\
(2) \ The graph $H_{2}^3$ is $3$-edge-connected, $\Z_4$-connected, but not $\Z_2^2$-connected.
\end{theorem}

\begin{proof}
(1) \ As $H_1^2$ is $\Z_2^2$-connected and, after contracting copies of $H_1^2$ in $H_{1}^3$,
the resulting graph is a singleton which is $\Z_2^2$-connected, we conclude by Lemma \ref{con}
that $H_1^3$ is $\Z_2^2$-connected. Since $H_1^3$ is obtained from $2$-sum operation of non-$\Z_4$-connected graphs,
Lemma \ref{2sum} shows that it is not $\Z_4$-connected.

It is also very straightforward to verify that $H_1^3$ is $3$-edge-connected.
Firstly, one can easily check that $H_1$ has only three trivial $2$-edge-cuts.
Secondly, the graph $H_1^2$, obtained from $2$-sum of $C_4$ and two copies of $H_1$,
has exactly three $2$-edge-cuts, each of which separates $z_1$ and $z_1'$.
At last, we can use these facts to show that $H_1^3$ is $3$-edge-connected as follows.
Specifically, the minimal degree of $H_1^3$ is three, so we only look at nontrivial edge-cuts.
If an edge-cut separates $z_k$ and $z_k'$ for some $k\in\{1,2,3\}$ in a copy of $H_1^2$,
then it has a size at least $3$ since we need at least two edges to separate $z_k$ and $z_k'$
in the copy of $H_1^2$  and there is a $z_kz_k'$-path outside that copy.
Assume instead, an edge-cut does not separate $z_k$ and $z_k'$ for any $k\in\{1,2,3\}$.
Then either it lies in the edges incident to $V(C_4)$, or it separates a copy of $H_1^2$
(where $z_k$ and $z_k'$ are in one component). In each case, the edge-cut must have a
size at least $3$. This proves that $H_1^3$ is $3$-edge-connected.

(2) \ The proof applies the same argument as (1) and thus omitted.

Now Theorem \ref{mainth} follows from Theorem \ref{THM: H13H23} and Lemma \ref{2sum}.
\end{proof}

\section{Constructions of cubic graphs}
\label{sec:cubic}

The constructions in this section rely on some basic properties of $K_4$ and $3$-prism (see Figure \ref{Fig:3prism}),
as shown in the following lemmas.

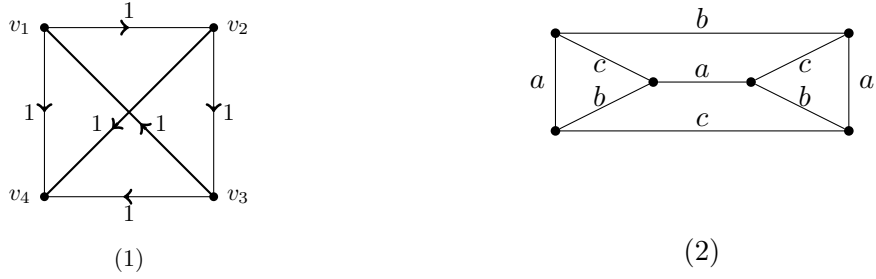
\begin{figure}

\minipage{0.4\textwidth}
\centering
\begin{tikzpicture}[scale=0.75]
\tikzstyle{every node}=[font=\small,scale=0.9]
\begin{scope}
[place/.style={thick,fill=black!100,circle,inner sep=1pt,minimum size=0.1mm,draw=black!100},
cloud/.style={thick,ellipse,draw=black!100,fill=white,inner sep=0pt,minimum size=20mm,minimum height=10mm}]

\draw (4,4) rectangle (1,1);
\node [place, label=left:$v_1$](a) at (1,4) {};
\node [place, label=right:$v_2$] at (4,4) {};
\node [place, label=right:$v_3$] at (4,1) {};
\node [place, label=left:$v_4$] at (1,1) {};
\draw[thick] (1,1) -- (4,4);\draw[<-,very thick](2.2,2.2) -- (2.3,2.3);
\coordinate [label=left:$1$] (c) at (2.2,2.3);
\draw[thick] (4,1) -- (1,4);\draw[<-,very thick](2.7,2.3) -- (2.8,2.2);
\coordinate [label=right:$1$] (c) at (2.8,2.3);
\draw[<-,very thick] (1,2.5) -- (1,2.6);
\coordinate [label=left:$1$] (a) at (1,2.5);
\draw[<-,very thick] (4,2.5) -- (4,2.6);
\coordinate [label=right:$1$] (b) at (4,2.5);
\draw[<-,very thick] (2.4,1) -- (2.5,1);
\coordinate [label=below:$1$] (c) at (2.5,1);
\draw[->,very thick] (2.4,4) -- (2.5,4);
\coordinate [label=above:$1$] (d) at (2.5,4);
\coordinate [label=above:$(1)$] (label) at (2.5,-0.5);
\end{scope}
\end{tikzpicture}
\endminipage\hfill
\minipage{0.6\textwidth}
\centering
\begin{tikzpicture}[scale=0.65]
\begin{scope}
[place/.style={thick,fill=black!100,circle,inner sep=1pt,minimum size=0.1mm,draw=black!100}]

\draw (6,2) rectangle (0,0);
\node [place] (c) at (6,0) {};
\node [place] (d) at (0,0) {};
\node [place] (a) at (6,2) {};
\node [place] (b) at (0,2) {};
\node [place] (e) at (2,1) {};
\node [place] (f) at (4,1) {};
\draw (e) to (f);\draw (e) to (d);\draw (e) to (b);\draw (a) to (f);\draw (c) to (f);
\coordinate [label=above:$b$] (label) at (3,1.9);
\coordinate [label=above:$a$] (label) at (3,0.9);
\coordinate [label=above:$c$] (label) at (3,-0.1);
\coordinate [label=right:$a$] (label) at (6,1);
\coordinate [label=left:$a$] (label) at (0,1);
\coordinate [label=above:$b$] (label) at (0.9,0.3);
\coordinate [label=above:$c$] (label) at (0.9,1);
\coordinate [label=above:$b$] (label) at (5.1,0.3);
\coordinate [label=above:$c$] (label) at (5.1,1);
\coordinate [label=above:$(2)$] (label) at (3,-3);
\end{scope}
\end{tikzpicture}
\endminipage\hfill

\caption{A $\Z_4$-flow of $K_4$ with boundary $1$ and a $3$-prism.}\label{Fig:K4}\label{Fig:3prism}
\end{figure}

\begin{lemma}\label{Prop:K_4}
Let $G$ be the complete graph $K_4$ with an orientation $D$. Define $\beta: V(G)\mapsto \{1\}$,
which is a zero-sum boundary function in $\ZA(G,\Z_4)$. Then for any mapping $\varphi:E(G)\mapsto \Z_4\setminus\{0\}$
with $\partial\varphi=\beta$, there exists a vertex $v$ of $G$ such that each edge $e=uv\in E(G)$
is either directed into $v$ with flow value $\varphi(e)=1$ or directed away from $v$ with flow value $\varphi(e)=3$.
\end{lemma}

\begin{proof}
Since $3=-1($mod $4)$, for convenience  we may assign the flow value of edges in $\{1,2\}$ and adapt an appropriate
orientation from $D$. By contradiction, suppose that there exists an orientation of $G$ and a
mapping $\varphi:E(G)\mapsto \{1,2\}$ with $\partial\varphi=\beta$ such that no vertex satisfies that
all incident edges are directed into it and with flow value $1$.
Since for any $v\in V(G)$, the degree of $v$ is $3$ and $\beta(v)=1$,
there is at least one edge $e$ assigned with flow value $\varphi(e)=1$.
By symmetry, assume $\varphi(v_1v_2)=1$ and the orientation is from $v_1$ to $v_2$ as in Figure \ref{Fig:K4} (1).
Since $\beta(v_2)=1$, we must have $\varphi(v_2v_3)=\varphi(v_2v_4)=1$ and
$v_2v_3,v_2v_4$ are all directed away from $v_2$. The similar assignments are applied for $v_3v_1$ and $v_3v_4$.
At last, we need only to assign the orientation and flow value of $v_1v_4$ to satisfy $\beta=\partial \varphi$.
We shall find that all the edges incident to $v_4$ are directed into $v_4$ with flow value $1$, a contradiction.
\end{proof}

A $3$-prism is a graph obtained by adding a perfect matching between two vertex-disjoint triangles (see Figure \ref{Fig:3prism}(2)).

\begin{lemma}\label{Prop:3prism}
The $3$-prism graph is unique $3$-edge-colorable. (That is, all proper $3$-edge-colorings $\phi:E(G)\mapsto \{a,b,c\}$
are isomorphic. See Figure \ref{Fig:3prism}(2).)
\end{lemma}
\begin{proof}
This fact is easy to observe and thus omitted.
\end{proof}

Now we shall prove Theorem \ref{mainthmcubic} with the following constructions.

\begin{theorem}\label{z22notz4in3.3}
Construct a graph $G$ by replacing every vertex of $K_4$ with a copy of $H_1$,
where every $2$-vertex in each copy is incident with an edge of $K_4$ (see Figure \ref{Fig:counterexample}).
Then the $3$-edge-connected cubic graph $G$ is $\Z_2^2$-connected, but not $\Z_4$-connected.
\end{theorem}

\begin{proof}
Clearly, $G$ is $3$-edge-connected. It follows from Lemma \ref{con}(1) that $C_2$ and $C_3$ are $\Z_2^2$-connected, thus by Lemma \ref{con}(2) $K_4$ is $\Z_2^2$-connected by contracting $3$-cycles and $2$-cycles consecutively.
By Lemma \ref{con} again, $G$ is $\Z_2^2$-connected since both $H_1$ and $K_4$ are $\Z_2^2$-connected.
We shall prove below that $G$ is not $\Z_4$-connected.
For $1\leq i\leq 4$, let $A_i$ be a copy of $H_1$, where the $2$-vertices of $A_i$ are $x_i$, $y_i$ and $z_i$ (see Figure \ref{Fig:counterexample}).
Since $H_1$ is not $\Z_4$-connected, there is a failed zero-sum boundary $\beta_1\in \ZA(H_1,\Z_4)$ such that
\begin{align}\label{ProofNoBetaFlowZ4}
  &\text{for any orientation of $H_1$,}\nonumber\\
  &\text{there is no mapping $\varphi:E(H_1)\mapsto \Z_4\setminus\{0\}$ such that $\partial\varphi=\beta_1$.}
\end{align}
Suppose, on the contrary, that $G$ is $\Z_4$-connected. Define $\beta: V(G)\mapsto \Z_4$ by
$$\beta(v)=\left\{
\begin{array}{ll}
\beta_1(v)-1  &\text{if}~v\in\{x_i,y_i,z_i|1\leq i\leq 4\};\\
\beta_1(v) &\text{otherwise}.
\end{array}
\right.
$$
Since $\sum_{v\in V(A_i)}\beta_1(v)\equiv0\pmod4$ for each $i$, we have
$$\sum_{v\in V(G)}\beta(v)=4\sum_{v\in V(A_1)}\beta_1(v)-12\equiv0 \pmod{4},$$
and so $\beta\in \ZA(G,\Z_4)$. Hence there is an orientation of $G$ and a mapping $f:E(G) \mapsto \Z_4\setminus\{0\}$ such that $\partial f=\beta$.

\begin{figure}
\begin{center}
\begin{tikzpicture}[scale=0.75]
\tikzstyle{every node}=[font=\small,scale=0.9]
\begin{scope}
[place/.style={thick,fill=black!100,circle,inner sep=1pt,minimum size=0.1mm,draw=black!100},
cloud/.style={thick,ellipse,draw=black!100,fill=white,inner sep=0pt,minimum size=28mm,minimum height=12mm}]

\node [cloud, rotate=45] (a1) at (1,5) {$A_1$};
\node [cloud, rotate=-45] (a3) at (1,1) {$A_3$};
\node [cloud, rotate=45] (a4) at (5,1) {$A_4$};
\node [cloud, rotate=-45] (a2) at (5,5) {$A_2$};
\node [place,label=left:$x_1$] (x1) at (1.7,3.7) {};
\node [place,label=right:$y_1$] (y1) at (2,4) {};
\node [place] (z1) at (2.25,4.3) {};\coordinate [label=above:$z_1$] (zz1) at (2.35,4.3);
\draw (z1) to (1.4,5.1);\draw (z1) to (1.6,5.3);
\draw (y1) to (1.2,4.9);\draw (y1) to (1,4.7);
\draw (x1) to (0.8,4.5);\draw (x1) to (0.6,4.3);
\node [place,label=left:$x_3$] (x3) at (1.7,2.3) {};
\node [place,label=right:$y_3$] (y3) at (2,2) {};
\node [place] (z3) at (2.25,1.7) {};\coordinate [label=below:$z_3$] (zz3) at (2.5,1.7);
\draw (z3) to (1.4,0.8);\draw (z3) to (1.6,0.6);
\draw (y3) to (1.2,1);\draw (y3) to (1,1.2);
\draw (x3) to (0.8,1.3);\draw (x3) to (0.6,1.5);
\node [place] (x2) at (3.7,4.3) {};\coordinate [label=above:$x_2$] (xx2) at (3.6,4.3);
\node [place,label=left:$y_2$] (y2) at (4,4) {};
\node [place,label=right:$z_2$] (z2) at (4.3,3.7) {};
\draw (z2) to (5.2,4.5);\draw (z2) to (5.4,4.3);
\draw (y2) to (4.8,4.9);\draw (y2) to (5,4.7);
\draw (x2) to (4.6,5.1);\draw (x2) to (4.4,5.3);
\node [place] (x4) at (3.65,1.7) {};\coordinate [label=below:$x_4$] (xx4) at (3.6,1.7);
\node [place,label=left:$y_4$] (y4) at (4,2) {};
\node [place,label=right:$z_4$] (z4) at (4.3,2.3) {};
\draw (z4) to (5.2,1.3);\draw (z4) to (5.4,1.5);
\draw (y4) to (4.8,1);\draw (y4) to (5,1.2);
\draw (x4) to (4.6,0.8);\draw (x4) to (4.4,0.6);
\draw (x4) to (z3);\draw (x1) to (x3);\draw (x2) to (z1);\draw (z4) to (z2);
\draw (y1) to (y4);\draw (y2) to (y3);
\end{scope}
\end{tikzpicture}
\end{center}
\caption{\small\it A $3$-edge-connected cubic graph that is $\Z_2^2$-connected, but not $\Z_4$-connected.}\label{Fig:counterexample}
\end{figure}
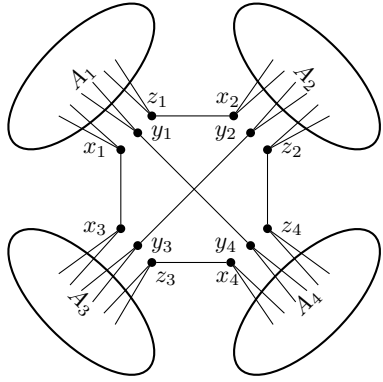

Consider the graph $F=G/\{\bigcup_{1\leq i\leq4} A_i\}$, which is a $K_4$. Suppose $w_i$ of $V(F)$ is the vertex corresponds to $A_i$. Let
$$\beta'(w_i)=\sum_{v\in V(A_i)}\beta(v)=\sum_{v\in V(A_i)}\beta_1(v)-3=1(\text{mod }4).$$
Denote $f'$ as the restriction of $f$ on $F$. Obviously, $\beta'$ is a zero-sum boundary of $F$ and $\partial f'=\beta'$.
By Lemma \ref{Prop:K_4}, there is a vertex $u$ in $F$ such that each incident edge of $u$ is either directed into $u$ with flow value $1$
or directed away from $u$ with flow value $3$.  Assume, without loss of generality, that the vertex $u$ corresponds to $A_1$ in $G$.

This implies that $\varphi=f|_{A_1}$, $f$ restricted to $A_1$, is a mapping such that $\partial\varphi=\beta_1$ by the definition of $\beta$,
which contradicts (\ref{ProofNoBetaFlowZ4}).
Hence $G$ is not $\Z_4$-connected.
\end{proof}

In the proof of Theorem \ref{z22notz4in3.3}, one may observe that the key ingredient is to apply Lemma \ref{Prop:K_4}
to show that the flow values outside a copy $A_i$ are uniquely determined, and so the flow restricted to $A_i$ satisfies
the failed zero-sum boundary, yielding a contradiction. The next construction is based on the same motivation,
for which we apply the property of $3$-prism in Lemma \ref{Prop:3prism} instead.

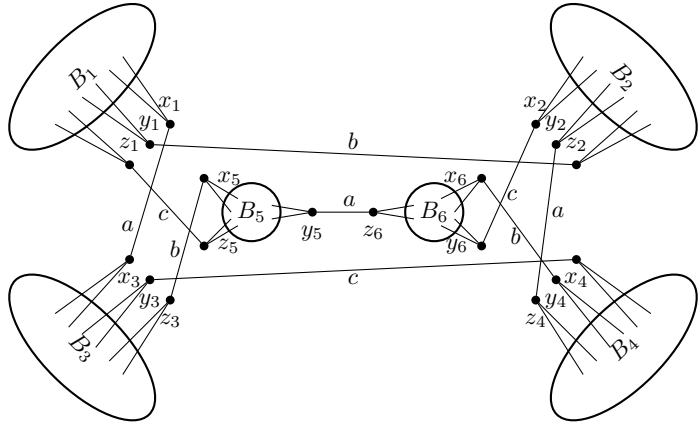
\begin{figure}[ht]
\begin{center}
\begin{tikzpicture}[scale=0.9]
\tikzstyle{every node}=[font=\small,scale=0.9]
\begin{scope}
[place/.style={thick,fill=black!100,circle,inner sep=1pt,minimum size=0.1mm,draw=black!100},
cloud/.style={thick,ellipse,draw=black!100,fill=white,inner sep=0pt,minimum size=28mm,minimum height=12mm},
cy/.style={thick,fill=black!0,circle,inner sep=3.5pt,minimum size=0.1mm,draw=black!100}]

\node [cloud, rotate=45] (a1) at (1,5) {$B_1$};
\node [cloud, rotate=-45] (a3) at (1,1) {$B_3$};
\node [cloud, rotate=45] (a4) at (9,1) {$B_4$};
\node [cloud, rotate=-45] (a2) at (9,5) {$B_2$};
\node [place,label=above:$z_1$] (z1) at (1.7,3.7) {};
\node [place,label=above:$y_1$] (y1) at (2,4) {};
\node [place,label=above:$x_1$] (x1) at (2.3,4.3) {};
\draw (x1) to (1.4,5.1);\draw (x1) to (1.6,5.3);
\draw (y1) to (1.2,4.9);\draw (y1) to (1,4.7);
\draw (z1) to (0.8,4.5);\draw (z1) to (0.6,4.3);
\node [place,label=below:$x_3$] (x3) at (1.7,2.3) {};
\node [place,label=below:$y_3$] (y3) at (2,2) {};
\node [place,label=below:$z_3$] (z3) at (2.3,1.7) {};
\draw (z3) to (1.4,0.8);\draw (z3) to (1.6,0.6);
\draw (y3) to (1.2,1);\draw (y3) to (1,1.2);
\draw (x3) to (0.8,1.3);\draw (x3) to (0.6,1.5);
\node [place,label=above:$x_2$] (x2) at (7.7,4.3) {};
\node [place,label=above:$y_2$] (y2) at (8,4) {};
\node [place,label=above:$z_2$] (z2) at (8.3,3.7) {};
\draw (z2) to (9.2,4.5);\draw (z2) to (9.4,4.3);
\draw (y2) to (8.8,4.9);\draw (y2) to (9,4.7);
\draw (x2) to (8.6,5.1);\draw (x2) to (8.4,5.3);
\node [place,label=below:$z_4$] (z4) at (7.7,1.7) {};
\node [place,label=below:$y_4$] (y4) at (8,2) {};
\node [place,label=below:$x_4$] (x4) at (8.3,2.3) {};
\draw (x4) to (9.2,1.3);\draw (x4) to (9.4,1.5);
\draw (y4) to (8.8,1);\draw (y4) to (9,1.2);
\draw (z4) to (8.6,0.8);\draw (z4) to (8.4,0.6);
\draw (x4) to (y3);\draw (x1) to (x3);\draw (y1) to (z2);\draw (z4) to (y2);
\node [cy] (a5) at (3.5,3) {$B_5$};
\node [place,label=below:$y_5$] (y5) at (4.4,3) {};
\draw (y5) to (3.8,2.9);\draw (y5) to (3.8,3.1);
\node [place,label=right:$z_5$] (z5) at (2.8,2.5) {};
\draw (z5) to (3.2,2.9);\draw (z5) to (3.3,2.8);
\node [place,label=right:$x_5$] (x5) at (2.8,3.5) {};
\draw (x5) to (3.2,3);\draw (x5) to (3.3,3.2);
\node [cy] (a6) at (6.2,3) {$B_6$};
\node [place,label=below:$z_6$] (z6) at (5.3,3) {};
\draw (z6) to (5.9,2.9);\draw (z6) to (5.9,3.1);
\node [place,label=left:$y_6$] (y6) at (6.9,2.5) {};
\draw (y6) to (6.5,2.9);\draw (y6) to (6.3,2.8);
\node [place,label=left:$x_6$] (x6) at (6.9,3.5) {};
\draw (x6) to (6.5,3);\draw (x6) to (6.3,3.2);
\draw (y5) to (z6);\coordinate [label=above:$a$] (ay5z6) at (4.95,2.95);
\draw (z5) to (z1);\coordinate [label=below:$c$] (cz1z5) at (2.2,3.17);
\draw (y6) to (x2);\coordinate [label=below:$c$] (cx2y6) at (7.35,3.47);
\draw (z3) to (x5);\coordinate [label=above:$b$] (bz3x5) at (2.37,2.2);
\draw (y4) to (x6);\coordinate [label=above:$b$] (by4x6) at (7.4,2.4);
\coordinate [label=above:$b$] (by1z2) at (5,3.8);
\coordinate [label=above:$c$] (cy3x4) at (5,1.8);
\coordinate [label=left:$a$] (ax1x3) at (1.9,2.8);
\coordinate [label=right:$a$] (ay2z4) at (7.8,3);
\end{scope}
\end{tikzpicture}
\end{center}
\caption{\small\it A $3$-edge-connected cubic graph that is $\Z_4$-connected, but not $\Z_2^2$-connected.}\label{Fig:Z2_2counterexample}
\end{figure}

Let $B_i(1\leq i\leq6$) be a copy of $H_2$, where the $2$-vertices of $B_i$ are $x_i$, $y_i$ and $z_i$.
\begin{theorem}
Assume that the $3$-prism is $3$-edge-colored with colors $a,b,c$. Let $(p_i,q_i,r_i)$ ($1\le i\le 6$) be all the permutations of $a,b,c$.
Replace each vertex of the $3$-prism with a copy $B_i$ of $H_1$, where the vertex-triple $(x_i, y_i, z_i)$ is identified with edges incident
to that vertex with color-triple $(p_i,q_i,r_i)$ for each $1\le i\le 6$. Let $G$ be the resulting graph. See Figure \ref{Fig:Z2_2counterexample}.
Then $G$ is $\Z_4$-connected, but not $\Z_2^2$-connected.
\end{theorem}

\begin{proof}
Since both $H_2$ and the $3$-prism are $\Z_4$-connected, the graph $G$ is $\Z_4$-connected by Lemma \ref{con}.
We shall show below that $G$ is not $\Z_2^2$-connected.
Note that for $\Z_2^2$-group connectivity, the orientation is irrelevant since each element is self-inverse.
Thus we will omit the statements of orientations. As $H_2$ is not $\Z_2^2$-connected,
there is a failed boundary $\beta_1\in \ZA(H_2,\Z_2^2)$ such that
\begin{equation}\label{ProofNoBetaFlowZ22}
  \text{there is no mapping $\varphi:E(H_2) \mapsto  \{(0,1),(1,0\
  ),(1,1)\}$ with $\partial\varphi=\beta_1$.}
\end{equation}

Define a function $\beta: V(G)\mapsto \Z_2^2$ as follows:

$$\beta(v)=\left\{
\begin{array}{ll}
\beta_1(v)-(0,1) & \text{if}~ v\in\{x_i|1\leq i \leq 6\};\\
\beta_1(v)-(1,0) & \text{if}~ v\in\{y_i|1\leq i \leq 6\};\\
\beta_1(v)-(1,1) & \text{if}~ v\in\{z_i|1\leq i \leq 6\};\\
\beta_1(v) & \text{otherwise}.
\end{array}
\right.
$$
Since $\sum_{v\in V(B_i)}\beta_1(v)=(0,0)$ in $\Z_2^2$ for each $1\leq i\leq 6$, we have
$$\sum_{v\in V(G)}\beta(v)= \sum_{i=1}^6\left(\sum_{v\in V(B_i)}\beta_1(v)\right)-6[(0,1)-(1,0)-(1,1) ]=(0,0) ~\text{in}~ \Z_2^2,$$
and thus $\beta\in\ZA(G,\Z_2^2)$.

By contradiction, suppose that $G$ is $\Z_2^2$-connected. So there is a mapping $f: E(G)\mapsto $ $\{(0,1),$ $(1,0),(1,1)\}$ such that $\partial f=\beta$.

Consider the graph $F=G/\{\bigcup_{1\leq i\leq6} B_i\}$, which is a $3$-prism. The flow $f$ restricted to it
provides a nowhere-zero $\Z_2^2$-flow, which is indeed a proper $3$-edge-coloring and the color-classes are precisely
the edges with values $(0,1),(1,0),(1,1)$, respectively. Hence the color-triple $(a,b,c)$ is a permutation of $(0,1),(1,0),(1,1)$.
Notice that edges incident to the triples of $\{x_i,y_i,z_i|1\leq i\leq 6\}$ for different $i$ are colored with different permutation
of color-set $\{a,b,c\}$. So each of the six permutations appears on exactly one vertex.
Hence there exists a triple $(x_k,y_k,z_k)$ corresponding to $((0,1),(1,0),(1,1))$, say $k=1$ without loss of generality.
That is $f(x_1x_3)=(0,1),$ $f(y_1z_2)=(1,0)$ and $f(z_1z_5)=(1,1)$. Now by definition of $\beta$,
the mapping $f$ restricted to $B_1$, $\varphi=f|_{B_1}$, is a mapping of $H_2$ such that $\partial\varphi=\beta_1$,
a contradiction to (\ref{ProofNoBetaFlowZ22}).
Therefore, $G$ is not $\Z_2^2$-connected.
\end{proof}

\section{Concluding Remarks}
\label{sec:collapsible}

\medskip
Theorem \ref{thmJLPT} of Jaeger et al. \cite{JLPT92} says that every $4$-edge-connected graph is $S$-connected for $|S|\ge 4$.
This particularly shows that group connectivity is equivalent for distinct groups of a same size for $4$-edge-connected graphs.
In fact, the graphs constructed in Theorems \ref{mainth} and \ref{mainthmcubic} are far from being $4$-edge-connected and
contain a lot of $3$-edge-cuts. It would be curious that whether lowing down the number of $3$-edge-cuts could guarantee the equivalence
relation of group connectivity.
\begin{problem}\label{prob-number3cut}
What is the maximum number $k$ such that, for all $3$-edge-connected graphs with at most $k$ $3$-edge-cuts,
$\Z_2^2$-connectivity and $\Z_4$-connectivity are equivalent?
\end{problem}

Note that, using a smaller $\Z_4$-connected non-$\Z_2^2$-connected graph obtained in Section 2 of \cite{Husek} (Figure 2 in that paper),
the smallest such $3$-edge-connected graphs that we can construct in Theorem  \ref{mainthmcubic} have $48$ edge-cuts of size three,
which shows $ k< 48$.

On the other hand, we provide a partial positive result from some known results on collapsible graphs (which are contractible graphs for Eulerian subgraph problem).
A graph $G$ is  {\it collapsible} if for any $N\subseteq V(G)$ of even order, there is a spanning connected subgraph of $G$ whose vertices have
degree exactly odd in $N$ and even otherwise. Lai \cite{LaiH99} showed that every collapsible graph is both $\Z_4$-connected and $\Z_2^2$-connected.
Moreover, it was proved in \cite{CHL1996} that every $3$-edge-connected graph with at most nine $3$-edge-cuts is collapsible,
and therefore, both $\Z_4$-connected and $\Z_2^2$-connected. Hence, we conclude that $$9\le k\le 47.$$
It would also be interesting to find the smallest $\Z_4$-connected non-$\Z_2^2$-connected graphs (with edge-connectivity $3$), \
and the other way around. This may help to solve Problem \ref{prob-number3cut}.

In this paper, Corollary \ref{Cor:3connn} completely answers the equivalence of group connectivity for $3$-edge-connected graphs.
The dual problem on graph coloring is still open, see \cite{LLSZ11}. Is it true that for distinct groups $S$ and $T$ with a same order,
$S$-group-colorability and $T$-group-colorability are equivalent (for simple graphs)?

\section*{Acknowledgments}
Miaomiao Han is partially supported by  National Natural Science Foundation of China (No. 11901434) and the Talent Fund Project of Tianjin Normal
University, China (No. 5RL159).
Jiaao Li is partially supported by  National Natural Science Foundation of China (No. 11901318), Natural Science Foundation of Tianjin (No. 19JCQNJC14100) and  the Fundamental Research Funds for the Central Universities, Nankai University (No. 63201147).
Xueliang Li and Meiling Wang are  partially supported by NSFC No. 11871034, 11531011 and NSFQH No. 2017-ZJ-790.

\footnotesize{

}

\end{document}